\documentclass[11pt, reqno]{amsart}
\usepackage{amsmath, amsthm, amsfonts, amsbsy, amssymb, upref, enumerate, bigstrut, booktabs, verbatim} 

\usepackage{pstricks}

\usepackage[authoryear, round]{natbib}

\usepackage{epsfig, graphicx}
\usepackage{pdfsync}
\usepackage[breaklinks]{hyperref}

\newtheorem{thm}{Theorem}[section]
\newtheorem{lmm}[thm]{Lemma}

\newcommand{\argmin}{\operatorname{argmin}}

\newcommand{\ee}{\mathbb{E}}

\newcommand{\pp}{\mathbb{P}}

\newcommand{\ra}{\rightarrow}
\newcommand{\rr}{\mathbb{R}}

\newcommand{\var}{\mathrm{Var}}
\newcommand{\ve}{\varepsilon}

\newcommand{\hm}{\hat{\mu}}

\newcommand{\TX}{\widetilde{X}}
\newcommand{\TK}{\widetilde{K}}
\newcommand{\bh}{\hat{\beta}}

\newcommand{\hs}{\hat{\sigma}}
\newcommand{\hl}{\hat{L}}

\begin{document}
\title[Regression and matrix estimation without tuning parameters]{High dimensional regression and matrix estimation without tuning parameters}
\author{Sourav Chatterjee}
\address{\newline Department of Statistics \newline Stanford University\newline Sequoia Hall, 390 Serra Mall \newline Stanford, CA 94305\newline \newline Email: \textup{\tt souravc@stanford.edu}}
\thanks{Research partially supported by NSF grant DMS-1441513}


\begin{abstract}
A general theory for Gaussian mean estimation that automatically adapts to unknown sparsity under arbitrary norms is proposed. The theory is applied to produce adaptively minimax rate-optimal estimators in high dimensional regression and matrix estimation that involve no tuning parameters.
%A new regression procedure for high dimensional data is proposed, which does not require the user to supply anything other than the design matrix and the response vector. The method is shown to be adaptively minimax rate-optimal for prediction error when the parameter vector is $\ell^1$-sparse. A generalization to arbitrary norm-induced notions of sparsity is provided.  The general theory is applied to give an adaptively minimax rate-optimal matrix estimator, where the sparsity of the unknown matrix is measured by its nuclear norm and the variance of the noise is unknown.
\end{abstract}

\maketitle
\setcounter{tocdepth}{1}
\tableofcontents

\section{Abstract theory}\label{general}
%The estimator $\bh$ is a special case of a general class of estimators, based on the following abstract theory. 
Suppose that $Y\sim N_n(\mu, \sigma^2I_n)$ for some $\mu\in \rr^n$ and $\sigma\ge 0$, where $I_n$ denotes the $n\times n$ identity matrix, and $N_n(\mu,\sigma^2I_n)$ is the $n$-dimensional Gaussian distribution with mean vector $\mu$ and covariance matrix $\sigma^2 I_n$. The statistical problem of Gaussian mean estimation is the problem of estimating the unknown mean vector $\mu$ using the observed data vector $Y$. Typically, the parameter $\sigma$ is also unknown. Given an estimator $\hm$, the most common measure of risk is the risk under the quadratic loss, namely, the quantity $\ee\|\hm-\mu\|^2$, where $\|\cdot \|$ denotes the Euclidean norm on $\rr^n$. The data vector $Y$ is itself an unbiased estimator of $\mu$. \cite{stein56} famously proved that the naive estimator $Y$ is inadmissible under quadratic loss, and an estimator that strictly dominates the naive estimator was produced by \cite{jamesstein61}.

A surprising number of problems in mathematical statistics can be framed as Gaussian mean estimation problems, where the aim is to construct estimators of $\mu$ that perform well when the true $\mu$ satisfies some given conditions. For example, it may be known to the statistician that the true $\mu$ belongs to some convex set $C$. A reasonable estimate of $\mu$ in this case is the Euclidean projection of the data vector $Y$ onto the set $C$. There is a wealth of literature on the analysis of this estimator and its applications. The monographs of \cite{vdvwellner}, \cite{vdgbook}, \cite{massart} and \cite{bg11} contain the essential references to the statistics literature on this topic. A precise approximation of the risk of this estimator under quadratic loss was recently obtained by \cite{chatterjee14b}.  The problem has also received considerable attention in the signal processing literature; see \cite{rv08}, \cite{stojnic09}, \cite{chandrasekaran}, \cite{oh13},  \cite{cj13}, \cite{amelunxen}, \cite{mccoytropp14a,mccoytropp14b} and \cite{foygelmackey14} for current developments on this front.

%Recall that a norm on $\rr^n$ is a function $K:\rr^n \ra [0,\infty)$ that has the following properties:
%\begin{enumerate}[\indent (i)]
%\item $K(a x) = |a| K(x)$ for all $x\in \rr^n$ and $a \in \rr$.
%\item $K(x+y)\le K(x)+K(y)$ for all $x,y\in \rr^n$.
%\item If $K(x)=0$ then $x=0$.
%\end{enumerate}

Often, the convex set $C$ that presumably contains the true $\mu$ appears in the form of a ball of some radius $r$ centered at the origin for some norm $K$ on $\rr^n$. Many high dimensional estimation problems of contemporary interest fit into this framework. We will see examples in later sections. To compute the projection estimator discussed in the preceding paragraph, the statistician needs to know the value of $r$. If $r$ is unknown, which is usually the case, then the projection estimator cannot be reliably defined. Using a wrong value of $r$ for projection can lead to spurious outcomes. The problem demands the construction of an adaptive estimator whose performance adapts to the unknown value of $r$, in the sense that the performance of the estimator should become better as the $K$-norm of the true $\mu$ decreases. %Adaptive estimation has a long history in the mathematical statistics literature, enjoying particular success in the domains of nonparametric function estimation and density estimation. The adaptive function estimator of \cite{ep84} is a striking example from the classical literature, which estimates an unknown function from noisy data without any input from the user about the possible degree of smoothness of the function. The monograph of \cite{bickeletal93} gives a thorough account of the developments in adaptive estimation until the early nineties. In the mid-nineties, \cite{dj94, dj95}  introduced a new approach to adaptive function estimation by combining the use of wavelets with the unbiased risk estimate of \cite{stein81}, leading to many important developments in the subsequent years.  Some recent developments in adaptive estimation for high dimensional data include a matrix estimation algorithm called USVT introduced by \cite{chatterjee15}, and a regression algorithm called SLOPE, introduced by \cite{bogdanetal14}. Most of these problems, under a Gaussian error assumptions, can be put into the framework of Gaussian mean estimation 

The goal of  this section is to propose a solution to this general problem. Namely, if $Y\sim N_n(\mu, \sigma^2I_n)$, and $K$ is a norm on $\rr^n$, to produce an estimator of $\mu$ that has good performance whenever $K(\mu)$ is small. The performance needs to get better as $K(\mu)$ gets smaller, and no knowledge about the value of $K(\mu)$ should be required to construct the estimator. 

 %The dual norm $K^\circ$ is sometimes also called the polar of $K$. 

%Take any norm $K$ on $\rr^n$. Let $Y\sim N_n(\mu, \sigma^2I_n)$. The problem is to estimate the mean vector $\mu$ using the response vector $Y$. Suppose that we require that the estimator has to have the property that it performs well when $K(\mu)$ is small. In the regression setting considered before, the mean vector is $\mu=X\beta_0$. An alternative way to state the assumption that $|\beta_0|_1$ is small is to say that $K(\mu)$ is small, where 
%\[
%K(\mu) = \min\{|\beta|_1:\beta\in \rr^p,\, \mu=X\beta\}\,.
%\]
%It is not difficult to show that $K$ is a norm when $X$ has full rank. Therefore the regression problem is a special case of this general problem. 

The theory presented in this section produces an estimator with the above property in this completely general setting. An application to high dimensional regression is given in Section \ref{regsec}, and an application to matrix estimation is given in Section \ref{matsec}. The abstract theory involves two main ideas. The first idea  shows how in the above framework, any estimator of $\sigma$ may be used to yield an estimate of the mean vector. This is the content of Theorem \ref{hsthm}, stated below. Note that the reverse direction is easy: given an estimate $\hat{\mu}$ of the mean vector $\mu$, one can easily get a reliable estimator $\hat{\sigma}$ as~$\hat{\sigma}^2 = \|Y-\hat{\mu}\|^2/n$. 

Recall that for a norm $K$, the dual norm $K^\circ$ is defined as
\begin{equation}\label{dualdef}
K^\circ(x) = \sup_{y\ne 0} \frac{x\cdot y}{K(y)}\,,
\end{equation}
where $x\cdot y$ denotes the standard inner product on $\rr^n$.
\begin{thm}\label{hsthm}
Let $K$ be a norm on $\rr^n$ and let $K^\circ$ be the dual norm of $K$. Take any $\mu\in \rr^n$, $\sigma \ge 0$, and let $Y\sim N_n(\mu, \sigma^2 I_n)$. Let $\hs$ be any random variable defined on the same probability space as $Y$. Define 
\[
\hm := \argmin\{ K(\nu): \nu\in \rr^n, \, \|Y-\nu\|^2\le n\hs^2\}\,.
\] 
Let $Z\sim N_n(0,I_n)$. Then 
\[
\frac{\ee\|\hm-\mu\|^2}{n\sigma^2} \le \frac{16 }{n\sigma}K(\mu)\ee(K^\circ(Z)) + 2\sqrt{\frac{2}{n}} + \frac{2\,\ee|\hs^2-\sigma^2|}{\sigma^2}\,,
\]
where $K^\circ$ is the dual norm of $K$.
\end{thm}
The theorem says that if $\hs$ is a good estimate of $\sigma$, and $K(\mu)$ is small, then $\hm$ is a good estimate of $\mu$. In particular, if the value of $\sigma$ is known, then we can just take $\hs = \sigma$ and make the third term in the error bound equal to zero.

The proof of Theorem \ref{hsthm} is given in Section \ref{proofs}. A rough sketch of the idea behind the proof is as follows. For each $L\ge 0$, let $B_L$ be the $K$-ball of radius $L$ centered at the origin, and let $\hm_L$ be the Euclidean projection of $Y$ on to $B_L$. Now take any $0\le L'\le L$, and consider the triangle formed by the vertices $Y$, $\hm_L$ and $\hm_{L'}$. It is a standard fact that if $y$ is the projection of a point $x$ on to a convex set $C$, and $z$ is any point in $C$, then the angle between the rays $yz$ and $yx$ is an obtuse angle. Therefore, since $\hm_{L'}\in B_L$ and $\hm_L$ is the projection of $Y$ on to $B_L$, the angle at $\hm_L$ in the above triangle must be an obtuse angle (see Figure \ref{fig1}). 

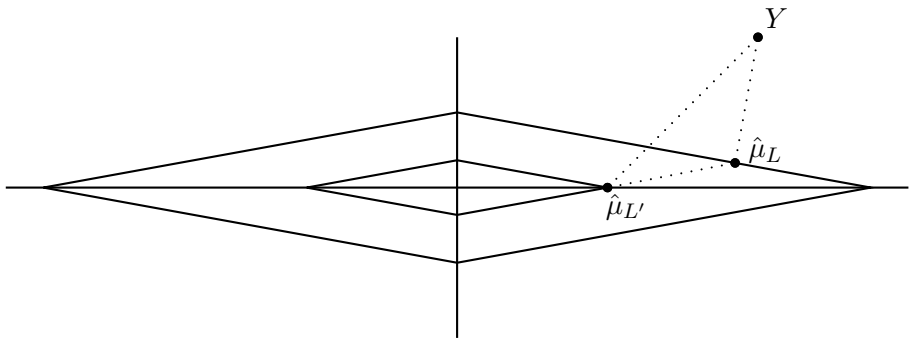
\begin{figure}[t!]
%\begin{pdfpic}
\begin{pspicture}(0,0)(12,5)
\psset{xunit=1cm,yunit=1cm}
\psline{-}(0,2)(12,2)
\psline{-}(6,0)(6,4)
\psline{-}(.5,2)(6,3)(11.5,2)(6,1)(.5,2)
\psline{-}(4,2)(6,2.363636)(8,2)(6,1.636364)(4,2)
\psline[linestyle = dotted]{*-*}(10,4)(8,2)
\psline[linestyle = dotted]{*-*}(10,4)(9.696, 2.328)
\psline[linestyle = dotted]{*-*}(9.696, 2.328)(8,2)
\rput(8.25, 1.75){$\hm_{L'}$}
\rput(10.25, 4.25){$Y$}
\rput(10.1, 2.5){$\hm_L$}
\end{pspicture}
\caption{The inner rhombus is the boundary of the $K$-ball $B_{L'}$ of radius $L'$ and the outer rhombus is the boundary of the $K$-ball $B_L$ of radius $L$. The point $\hm_L$ is the projection of $Y$ on to $B_L$ and the point $\hm_{L'}$ is the projection of $Y$ on to $B_{L'}$. Geometric considerations imply that in the triangle formed by the dotted lines, the angle at $\hm_L$ is necessarily obtuse. This, in turn, implies that if the two rays emanating from $Y$ are approximately of equal length, then $\hm_L$ must be close to $\hm_{L'}$.} 
\label{fig1}
\end{figure}

Since the angle at $\hm_L$ is an obtuse angle, an easy geometric argument shows that $\hm_{L'}$ is close to $\hm_L$ if and only if 
\[
\|Y-\hm_{L'}\|\approx \|Y-\hm_L\|\,,
\]
where $a\approx b$ means that $a/b$ is close to $1$. 

The next step in the proof is to argue that $\|Y-\hm\|\le \|Y-\nu\|$ for any $\nu$ with $K(\nu)\le K(\hm)$. This is easy to establish using the definition of $\hm$. This shows that $\hm= \hm_{\hat{L}}$, where $\hat{L}=K(\hm)$.  Moreover, it is also easy to show that 
\[
\|Y-\hm\|^2 = n\hs^2\,.
\]
The final step is to show that if $K(\mu)$ is small, and $L:= K(\mu)$, then $\hm_L$ is close to $\mu$. This is well known argument, often used in proving error bounds for penalized regression estimators. As a consequence, we get 
\[
\|Y-\hm_L\|^2 \approx \|Y-\mu\|^2\,.
\]
But $\|Y-\mu\|^2\approx n \sigma^2$ by the law of large numbers, and by the second step of the proof, if $\hs$ is a good estimate of $\sigma$, then 
\[
\|Y-\hm_{\hat{L}}\|^2 \approx n \sigma^2\,.
\]
Combining, we get 
\[
\|Y-\hm_{\hat{L}}\| \approx \|Y-\hm_L\|\,. 
\]
The first step of the proof now implies that $\hm_{\hat{L}}$ is close to $\hm_L$. But $\hm = \hm_{\hat{L}}$ and $\hm_L$ is close to $\mu$. Thus, $\hm$ is close to $\mu$. 

Although Theorem \ref{hsthm} gives a prescription for constructing an estimator of $\mu$ using an estimator of $\sigma$, it does not solve the problem of estimating $\sigma$ if $\sigma$ is unknown. It turns out that it is possible to directly construct an estimator of $\sigma$ that has good performance whenever $K(\mu)$ is sufficiently small. This is second key idea of the theory presented in this section, which yields the next theorem.
%It is somewhat surprising that such an estimate can be constructed without first having to perform a regression analysis and obtaining residuals, since the standard (and possibly only known) way to estimate $\sigma$ is through residual sum-of-squares. 
\begin{thm}\label{shthm}
Take any $n\ge 5$. Suppose that $K$ is a norm on $\rr^n$ and $Y\sim N_n(\mu, \sigma^2I_n)$ for some $\mu\in \rr^n$ and $\sigma\ge 0$. Let $Z\sim N_n(0,I_n)$ be defined on the same probability space as $Y$. Define 
\[
\hs := \frac{K(Y)}{K(Z)}\,.
\] 
Let $K^\circ$ be the dual norm of $K$. For each $k\ge 1$, let $m_k := (\ee(K^\circ(Z)^k)^{1/k}$. Let 
\[
a := \sup\{\|v\|: v\in \rr^n,\, K^\circ(v)\le 1\}\,.
\] 
Then
\[
\ee(\hs-\sigma)^2 \le \frac{K(\mu)^2 m_2^2}{(n-2)^2} + \frac{32\sqrt{2}\sigma^2 a^2m_4^2}{(n-4)^2}\,.
\]
\end{thm}
This result says, roughly speaking, that whenever $K$ satisfies certain conditions and $K(\mu)$ is small enough, then $\hs$ is a good estimate of $\sigma$. As usual, the proof is given in Section \ref{proofs}. A brief sketch of the idea behind the proof is as follows. Let $\ve := Y-\mu$. If $K(\mu)$ is small, then by the triangle inequality for $K$,  
\[
K(Y)= K(\ve + \mu)\approx K(\ve)\,.
\]
\begin{comment}
(See Figure \ref{fig2} for a pictorial representation.) 

\begin{figure}[t!]
%\begin{pdfpic}
\begin{pspicture}(0,0)(12,5)
\psset{xunit=1cm,yunit=1cm}
\psline{-}(0,2)(12,2)
\psline{-}(6,0)(6,4)
\psline{-}(0,2)(6,3.090909)(12,2)(6,.909091)(0,2)
\psline{-}(4,2)(6,2.363636)(8,2)(6,1.636364)(4,2)
\psline[linestyle = dotted]{*-*}(6,2)(7.5,2.090909)
\psline[linestyle = dotted]{*-*}(7.5,2.090909)(7.3,2.85454545)
\psline[linestyle = dotted]{*-*}(6,2)(5.8, 2.763636)
\psline[linestyle = dotted]{-}(1.6,2)(6, 2.8)(10.4,2)(6,1.2)(1.6,2)
\rput(7.35, 3.2){$Y$}
\rput(7.7,2.25){$\mu$}
\rput(5.65, 2.6){$\ve$}
\end{pspicture}
\caption{The rhombi represent boundaries of $K$-balls. The innermost rhombus has radius $K(\mu)$, the middle rhombus has radius $K(\ve)$, and the outer rhombus has radius $K(Y)$, all with respect to the norm $K$. The picture demonstrates how smallness of $K(\mu)$ implies that $K(Y)\approx K(\ve)$.} 
\label{fig2}
\end{figure}
\end{comment}
The second step is to observe that 
\[
K(Z) = \sup\{Z\cdot v: v\in \rr^n,\, K^\circ(v)\le 1\}\,,
\]
and therefore $K(Z)$ is the maximum of a Gaussian field. There are general inequalities which show that under mild conditions, the fluctuations of the maximum of a Gaussian field are small compared to its expected value. If these conditions hold, then $K(Z)$ would be close to $\ee(K(Z))$ with high probability. By the same logic,  since $\ve \sim N_n(0,\sigma^2I_n)$, therefore 
\[
K(\ve) \approx \sigma \ee(K(Z))
\]
with high probability. Combining this with the first step of the proof, we get 
\[
\hs = \frac{K(Y)}{K(Z)}\approx \frac{K(\ve)}{K(Z)}\approx \sigma\,. 
\]
It is easy to see how a  combination of Theorems \ref{hsthm} and~\ref{shthm} can give an estimator of $\mu$ that has accuracy whenever $K(\mu)$ is small enough. Moreover, two different norms can be used for the two parts. An example of a result that is obtained by combining Theorem \ref{hsthm} and Theorem \ref{shthm} is the following.  
\begin{thm}\label{genthm}
Take any $n\ge 5$. Suppose that $K$ and $\widetilde{K}$ are two norms on $\rr^n$ and $Y\sim N_n(\mu, \sigma^2I_n)$ for some $\mu\in \rr^n$ and $\sigma\ge 0$. Let $Z\sim N_n(0,I_n)$ be defined on the same probability space as $Y$. Define
\[
\hs := \frac{\widetilde{K}(Y)}{\widetilde{K}(Z)}
\]
and 
\[
\hm := \argmin\{ K(\nu): \nu\in \rr^n, \, \|Y-\nu\|^2\le n\hs^2\}\,.
\] 
For each $k\ge 1$, let $m_k := (\ee(\widetilde{K}^\circ(Z)^k)^{1/k}$, where $\widetilde{K}^\circ$ is the dual norm of $\widetilde{K}$. Let 
\[
a := \sup\{\|v\|: v\in \rr^n,\, \widetilde{K}^\circ(v)\le 1\}\,.
\] 
Then
\begin{align*}
\frac{\ee\|\hm-\mu\|^2}{n\sigma^2} &\le \frac{16}{n\sigma}K(\mu)\ee(K^\circ(Z)) + 2\sqrt{\frac{2}{n}}  +  \frac{2\widetilde{K}(\mu)^2 m_2^2}{(n-2)^2\sigma^2}\\
&\qquad  + \frac{64\sqrt{2} a^2m_4^2}{(n-4)^2} + \frac{4\widetilde{K}(\mu) m_2}{(n-2)\sigma} + \frac{28am_4}{n-4}\,.
\end{align*}
\end{thm}
It is not clear whether the estimator $\hm$ given in  Theorem \ref{genthm} has any general optimality property under any set of conditions. However we will see in the following sections that in special cases of interest, $\hm$ indeed turns out to be minimax rate-optimal.  

%In the absence of a general lower bound, the fact that  $\hm$ turns out to be minimax rate-optimal in the special case of Theorem \ref{mainthm} appears to be a fortunate coincidence. It would be very interesting if an abstract minimax lower bound can be proved in the setting of Theorem \ref{genthm}.

%When $X$ has rank $n$, Theorem \ref{genthm} implies Theorem \ref{mainthm}  upon taking $\widetilde{K}$ to be the norm 
%\[
%\widetilde{K}(\mu) = \min\{|\beta|_1: \beta\in \rr^{p+n},\, \mu=\TX\beta\}
%\]
%and  $K$ to be the norm 
%\[
%K(\mu) = \min\{|\beta|_1: \beta\in \rr^{p},\, \mu=X\beta\}\,.
%\] 
%The reason for using two different norms is that the quantity $a$ in Theorem \ref{genthm} is well-behaved for $\TK$ but may not be well-behaved for $K$.  When the rank of $X$ is less than $n$, a slightly modified argument is required. The details of these arguments are given in Section \ref{proofs}.

\section{Application to high dimensional regression}\label{regsec}
Consider the familiar regression framework. Let $n$ and $p$ be two positive integers, and let $X$ be an $n\times p$ matrix with real entries, called the design matrix. Let $\beta_0\in \rr^p$ be a vector of parameters, $\sigma\ge 0$ be another parameter, and let
\begin{equation}\label{regeq}
Y = X\beta_0+\ve\,,
\end{equation}
where $\ve \sim N_n(0,\sigma^2I_n)$. The vector $\beta_0$ and the number $\sigma$ are unknown to the statistician, who knows only the response vector $Y$ and the design matrix~$X$. The  objective is to estimate the vector $\beta_0$ from the observed vector $Y$, which is called the response vector. 

The regression problem is called high dimensional if the number of covariates $p$ is large, usually much larger than the number of data points $n$. Regression problems where the number of covariates exceeds the number of responses have become increasingly important in the last twenty years. Statisticians have devised a number of penalized regression techniques for dealing with such problems, such as the Lasso by \cite{tibs96}, basis pursuit by \cite{chenetal98}, the SCAD algorithm of \cite{fanli01}, the LARS algorithm of \cite{efronetal04}, the elastic net by \cite{zh05} and the Dantzig selector by \cite{candestao07}. Sophisticated variants of these methods have emerged over the years, such as the group Lasso by \cite{yuanlin06}, the adaptive Lasso by \cite{zou06} and the square-root Lasso by~\cite{bcw11}. Penalized regression is not the only approach; for example, methods of model selection by testing hypotheses have been proposed by \cite{birgemassart01}, \cite{birge06},  \cite{abramovichetal06} and \cite{barbercandes15}. 

There is now a large body of literature on the theoretical properties of the penalized regression techniques mentioned above. Perhaps the most widely studied among these methods is the Lasso. The theoretical analysis of basis pursuit by \cite{chenetal98} already had substantial information in this direction because of the close connection between the Lasso and basis pursuit. \cite{knightfu00} analyzed the Lasso when $p$ is fixed and $n\ra \infty$. Later, theoretical results that accommodated $p$ growing as fast as $n$ or even faster, began to emerge. A notable early example is \cite{gr04}, where the influential notion of `persistence' for measuring the efficacy of a high dimensional estimator was introduced. \cite{meinshausenbuhlmann06} gave a novel application of the Lasso to recovering dependency structures in a paper that contains some of the earliest theoretical techniques for analyzing the Lasso in a high dimensional setting. \cite{btw07a} were the first to prove oracle inequalities for $\ell^1$-penalized regression.   The Dantzig selector was introduced by  \cite{candestao07}, who also introduced a number of novel ideas that spurred much of the subsequent research on the analysis of sparse regression techniques. \cite{zhaoyu07}, introduced the `irrepresentability condition' that guarantees consistent model selection by the Lasso. This work was significantly refined and extended by \cite{zhanghuang08}, \cite{meinshausenyu09} and \cite{wainwright09}. 

The optimality of model selection by the Lasso under a wide range of conditions was established by \cite{candesplan09}. Theoretical analysis of the Lasso for generalized linear models was carried out by \cite{vdg08}. \cite{bickeletal09} constructed a unified theoretical framework that is capable of analyzing the Lasso, the Dantzig selector, and a variety of other techniques. Unification of oracle inequalities for a variety of high dimensional techniques, including the Lasso and the Dantzig selector, was achieved by \cite{vdgb09}. A further unification with general high dimensional $M$-estimators was obtained by \cite{negahbanetal12}. \cite{raskuttietal11}, \cite{bartlettetal12} and \cite{chatterjee14, chatterjee14b} revisited prediction error bounds for the Lasso. A nice textbook reference for many of these developments is \cite{bg11}. 

Recently, researchers have started investigating the asymptotic distributional properties of penalized regression estimates and ways to build confidence intervals and carry out tests of hypotheses; see \cite{wr09}, \cite{mmb09}, \cite{mb10}, \cite{mtc11}, \cite{tt12}, \cite{vdgb14}, \cite{jm14a, javanmardmontanari14}, \cite{zhangzhang14} and \cite{lttt14} for the latest developments.

A common feature of  the methods discussed above is that the user is required to supply the values of one or more `tuning parameters'. The tuning parameters are numbers or other objects that are chosen at the user's discretion, and the user `tunes' the values of these parameters to get optimal results. In practice, the tuning process usually involves the data, for example through cross-validation. The resultant estimates are invariably a highly complicated objects with very little theoretical understanding. The only existing papers that provide some level of theoretical justification for cross-validated estimates in high dimensional regression are those of \cite{lecuemitchell12},  \cite{hm13a, hm13b, hm14} and \cite{cj15}. There is a closely related technique, known as `aggregation', that has a much more extensive theoretical foundation. The central goal of aggregation is to take a finite collection of predictors, and find a combination of them that optimizes some measure of performance.  This combination itself may depend on the data. Aggregation was pioneered by \cite{nemirovski00}, and vastly developed by many authors, including \cite{yang00, yang01, yang04}, \cite{gyorfietal02}, \cite{wegkamp03}, \cite{tsybakov04}, \cite{catoni04}, \cite{btw07b} and \cite{rt07, rt12}. Aggregation, however, has its own tuning parameter that requires user intervention, namely, the choice of predictors that are aggregated.

%However, the state of the theory is unsatisfactory.%How the user should optimize is usually not clearly prescribed. 

%The vast majority of theoretical results in this subject assume that the tuning parameters are constants that are chosen by the user in a manner that is independent of the data. This leads to an unbridgeable gap between theory and practice, because in practice, the values of the tuning parameters are almost always chosen using the data. For example, one of the most common approaches for optimizing over the values of tuning parameters is to use cross-validation. Once the tuning parameters are chosen using the data, the resulting estimates are completely different mathematical objects than the estimates analyzed in the theoretical literature, where the tuning parameters are assumed to be constants chosen in a data-independent manner. This is not merely a technical point, because the regression estimates are typically quite sensitive to small changes in the values of the tuning parameters, which implies that the method of choosing these values is a matter of consequence. Even in methods that theoretically validate optimization over tuning parameters, such as various aggregation methods, the user has to choose something --- for example, the set of predictors to optimize over. Therefore the current state of affairs is that in spite of a large body of impressive mathematical results, the high dimensional regression techniques used by applied statisticians in real life  have very little theoretical backing. 

The mathematical statistics literature cited above contains prescriptions for choosing optimal values of the tuning parameters. For example, if $\sigma$ is known, a popular choice for the penalty parameter in ordinary  Lasso is 
\[
\lambda = \frac{4\|X^Tw\|_\infty}{n}\,,
\]
where $X^T$ is the transpose of the design matrix $X$, $w$ is a vector of $n$  i.i.d.~$N(0,\sigma^2)$ random variables, and $\|\cdot \|_\infty$ denotes the $\ell^\infty$ norm on $\rr^n$. See \cite{negahbanetal12} and \cite{clw} for further details.

The main problem with such theoretical prescriptions is that they require a priori knowledge about the unknown parameter $\sigma$. A variety of estimates for $\sigma$ in high dimensional regression have been proposed in recent times, for example by \cite{stadleretal10}, \cite{fanetal12}, \cite{sunzhang12}, \cite{dicker14} and \cite{cj15}. The difficulty with using these estimates  is that these estimates themselves involve tuning parameters, and moreover they usually need the regression to be performed before the estimate of $\sigma$ can be produced.  One regression technique where knowledge about $\sigma$ is not required is the square-root Lasso of \cite{bcw11}. Square-root Lasso has many nice properties, such as near-oracle performance when the number of nonzero components of $\beta_0$ is small. However, square-root Lasso is not completely free of tuning parameters: it has its own tuning parameter that needs to be calibrated by the user.%It seems, however, that  theoretical suggestions for tuning parameter values are rarely used in practice.  

The main contribution of this section is a simple new regression procedure for high dimensional data that does not need the user to input anything other than the design matrix and the response vector.  The estimator is obtained by specializing the abstract method proposed in Section \ref{general} to the regression setting. A general error bound is provided in Theorem \ref{mainthm} below. The error bound implies that under certain conditions, the estimate is adaptively minimax rate-optimal in $\ell^1$ balls.  %A more general estimator is proposed later in Subsection \ref{general}, where $\ell^1$ is replaced by an arbitrary norm.

%Simulation results, presented in Subsection \ref{simulsec}, demonstrate that when compared with the most popular practical algorithm --- the Lasso with cross-validation --- the proposed method typically has larger prediction error but is better at model selection. However, more extensive studies are needed to give a final verdict on the usefulness and real-life performance of the proposed estimator. 

%The proposed estimator demonstrates that tuning parameters are not absolutely necessary for high dimensional regression; it is possible to estimate adaptively. Adaptive estimators have been constructed in other contexts such as density estimation and function estimation (see literature review below), but no such method used to exist for high dimensional linear regression. Hopefully, even if it turns out that the method proposed here does not perform exceedingly well on real data, this may open the doors to more inventive and useful adaptive regression techniques in the future. 

%\subsection{The proposed estimator}\label{regest}
Given a response vector $Y$ and a design matrix $X$ related by equation  \eqref{regeq}, the proposed technique produces a randomized estimate $\bh$ of $\beta_0$ and a randomized estimate $\hs$ of $\sigma$ through the following sequence of steps. The method requires no conditions on the design matrix, but seems to perform better in simulations if the design matrix is standardized.% The only assumption is that $X$ has at least one nonzero entry.%, assuming that the design matrix $X$ has rank $n$. If the rank is less than $n$, one can always add extra covariates to artificially increase the rank to $n$.
\vskip.2in
\hrule
\begin{enumerate}[\indent Step 1:]
\item Let $\gamma := \max_{1\le j\le p} \|X_j\|/\sqrt{n}$, 
where $X_j$ is the $j^{\mathrm{th}}$ column of $X$ and $\|X_j\|$ is the Euclidean norm of $X_j$. 
\item Define a new matrix $\TX$ of order $n\times(p+n)$ as $\TX := [X \  \sqrt{n}\gamma I_n]$.  %Note that $\TX$ is necessarily of full rank, even if $X$ is not.
\item Generate $Z\sim N_n(0,I_n)$, independent of $Y$. 
\item Let $M_1 := \min\{|\beta|_1:\beta\in \rr^{p+n},\, Y=\TX \beta\}$ 
and  $M_2 := \min\{|\beta|_1:\beta\in \rr^{p+n},\, Z=\TX \beta\}$, where $|\beta|_1$ denotes  the $\ell^1$ norm of $\beta$. 
\item Let $\hs := M_1/M_2$. 
\item Let $Y'$ be the projection of $Y$ on to the column space of $X$.
\item 
Let $\bh := \argmin\{|\beta|_1: \beta\in \rr^{p},\, \|Y'-X\beta\|^2 \le k \hs^2\}$, where $k$ is the rank of $X$. 
%\item Let $u$ be the Euclidean projection of $\TX\beta^*$ on to the column space of $X$.
%\item Let $\bh := \argmin\{|\beta|_1:\beta\in \rr^p,\, u=X\beta\}$.
\end{enumerate}
\hrule
\vskip.2in
%The estimate $\bh$ is not hard to compute using standard convex optimization routines. In fact, the algorithm used for computing Lasso estimates may be used in a certain way to compute $\bh$. This will be explained in more detail in Subsection \ref{simulsec}. 
The theorem stated below gives an upper bound on the expected mean squared prediction error of $\bh$. The prediction error measures how well $X\bh$ estimates the mean vector $X\beta_0$. It also gives an upper bound on the risk of $\hs$ under quadratic loss. %The usual definition of this error is $\ee\|X\bh - X\beta_0\|^2/n$. In this article, however, the prediction error of $\bh$ is defined as $\ee\|X\bh - X\beta_0\|^2/n\sigma^2$. The purpose of dividing by $\sigma^2$ is twofold. First, it makes the error unit-free. That is, if the unit in which the data is measured is changed, it does not affect the value of this prediction error. Second, the data vector $Y$ is itself a na\"ive estimate of the mean vector $X\beta_0$, and $\ee\|Y-X\beta_0\|^2=n\sigma^2$. Therefore the quantity $\ee\|X\bh - X\beta_0\|^2/n\sigma^2$ measures how much better $X\bh$ estimates the mean vector $X\beta_0$, compared to the na\"ive estimate $Y$. 
\begin{thm}\label{mainthm}
Let all notation be as above. Suppose that $n\ge 8$ and $p\ge 8$. Let
\[
r := \frac{|\beta_0|_1\gamma}{\sigma} \sqrt{\frac{\log (p+n)}{n}}\,.
\]
Then 
\[
\frac{\ee\|X\bh-X\beta_0\|^2}{n\sigma^2} \le Cr + Cr^2 + C\sqrt{\frac{\log (p+n)}{n}}+ \frac{C\log (p+n)}{n}\,
%\frac{\ee\|X\bh-X\beta_0\|^2}{n\sigma^2} \le 24 r + 45r^2 + 140\sqrt{\frac{\log (p+n)}{n}} + 4526\,\frac{\log (p+n)}{n} + \sqrt{\frac{2}{n}}\,.
\]
and
\[
\ee\biggl(\frac{\hs}{\sigma}-1\biggr)^2 \le Cr^2 + \frac{C\log (p+n)}{n}\,,
\]
where $C$ is a universal constant.
\end{thm}
%A noteworthy feature of this result is that the error bound depends on the design matrix $X$ only through the number $\gamma$. 
The reason for dividing the risk by $n\sigma^2$ is to compare $X\bh$ with the naive unbiased estimate of $X\beta_0$, which  is simply the vector $Y$. The risk of the naive estimator is equal to $n\sigma^2$. 

The proof of Theorem \ref{mainthm} is based on a straighforward application of Theorem \ref{genthm} of Section \ref{general}. When $X$ has rank $n$, Theorem \ref{genthm} implies Theorem \ref{mainthm}  upon taking $\widetilde{K}$ to be the norm 
\[
\widetilde{K}(\mu) = \min\{|\beta|_1: \beta\in \rr^{p+n},\, \mu=\TX\beta\}
\]
and  $K$ to be the norm 
\[
K(\mu) = \min\{|\beta|_1: \beta\in \rr^{p},\, \mu=X\beta\}\,.
\] 
The reason for using two different norms is that the quantity $a$ in Theorem \ref{genthm} is well-behaved for $\TK$ but may not be well-behaved for $K$.  When the rank of $X$ is less than $n$, a slightly modified argument is required. The details of these arguments are given in Section \ref{proofs}.

The minimax rate-optimality of $\bh$ for prediction error in $\ell^1$ balls is established using Theorem~3 in  \cite{raskuttietal11}. This result says the following. Let all notation be as above. Take any $L\ge 0$. Suppose that there are constants $\kappa > 0$ and $c\ge 0$ such that for all $\beta\in \rr^p$ with $|\beta|_1\le 2L$, 
\begin{equation}\label{rascond}
\frac{\|X\beta\|}{\sqrt{n}}\ge \kappa\|\beta\| - \kappa c L \sqrt{\frac{\log p}{n}}\,.
\end{equation}
Suppose further that $c_1, c_2> 0$ and $\delta\in (0,1)$ are constants such that
\begin{equation}\label{rascond2}
\frac{p}{L\sqrt{n}} \ge c_1p^\delta \ge c_2\,.
\end{equation}
Then Theorem 3 of \cite{raskuttietal11} implies that
\begin{align}\label{raskutti}
\inf_{\tilde{\beta}} \sup_{\beta_0\,:\, |\beta_0|_1\le L}\frac{\ee\|X\tilde{\beta}-X\beta_0\|^2}{n\sigma^2} \ge  \frac{C L}{\sigma} \sqrt{\frac{\log p}{n}}\,,
\end{align}
where the infimum is taken over all possible estimators $\tilde{\beta}$ and $C$ is a constant that depends only on $\gamma$, $\kappa$, $c$, $c_1$, $c_2$ and $\delta$.

It was shown by \cite{raskuttietal11} and  \cite{ruzho13} that the condition \eqref{rascond} holds for a large class of design matrices. This includes, but is not limited to, matrices with i.i.d.~Gaussian entries. The condition~\eqref{rascond2} is a mild growth condition on the dimensions of the design matrix which has nothing to do with the matrix entries. It is possible that the lower bound~\eqref{raskutti} holds under more general conditions than \eqref{rascond} and \eqref{rascond2}.

The square-root Lasso of \cite{bcw11} also achieves minimax rate-optimality (and moreover, near-oracle performance) with a theoretically prescribed choice of the tuning parameter. However, this optimality holds when $\beta$ has a small number of nonzero entries, and not when $\beta$ has small $\ell^1$ norm.

%An important note about the prediction error is that if $X$ has rank $n$, one can always define a naive estimator $\tilde{\beta}$ as an arbitrary solution of  $Y=X\tilde{\beta}$; with this choice, $\ee\|X\tilde{\beta}-X\beta_0\|^2/n\sigma^2$ is exactly equal to $1$, irrespective of the value of $\beta_0$. Therefore the minimax error  in any $\ell^1$ ball can be at most $1$. This implies that the lower bound \eqref{raskutti} must necessarily be $\le 1$ in its region of validity.  

%Comparing the lower bound \eqref{raskutti} with the upper bound from Theorem~\ref{mainthm} shows that the upper bound on the prediction error of the estimator $\bh$ matches the lower bound in any $\ell^1$ ball up to a constant factor, provided that the lower bound is $\le 1$. If the radius of the ball is so large that the lower bound is greater than $1$, then we are in an uninteresting region where no estimator can perform better than the naive estimator, as discussed in the previous paragraph. In other words, the proposed estimator is minimax rate-optimal in any $\ell^1$ ball whose radius is so small that the minimax error is smaller than the error of the naive estimator. One may say that this is the `nontrivial region' for estimation.

Two other recently proposed adaptively rate-optimal estimators for high dimensional regression deserve mention. One, called SLOPE, was suggested by \cite{bogdanetal14}. SLOPE  was shown to be adaptively minimax rate-optimal by \cite{sucandes15} for a certain norm called the `sorted $\ell^1$ norm'. Another one, called $\mathrm{AV}_\infty$, was suggested by \cite{clw} and shown to be adaptively rate-optimal for sup-norm sparsity.

%\subsection{Issues}
There are several unresolved issues about the estimator proposed in this section. The foremost theoretical issue is that Theorem \ref{mainthm} requires the Gaussian error assumption. Although simulation results suggest that some version of the theorem should hold even for non-Gaussian errors, there is no mathematical proof. The main difficulty in extending the result to the non-Gaussian setting is that one has to show that the quantity $M_2$ remains relatively unchanged if $Z$ is replaced by a non-Gaussian vector with i.i.d.~components that have zero mean and unit variance. This universality cannot be expected in full generality (for example, it fails when $X$ is the identity matrix), but some mild condition on $X$ may suffice. Homoskedasticity is another assumption that one should be able to drop, although removing this assumption will probably require a modification of the estimator. Another minor problem is that the proposed estimator is a randomized estimator. Although most estimators used in practice are randomized because the tuning parameters are chosen using randomized processes such as cross-validation, it would be nice to have a non-randomized estimator with properties similar to that of the proposed one.

\section{Application to matrix estimation}\label{matsec}
Let $M = (\mu_{ij})_{1\le i\le l, \, 1\le j\le m}$ be an $l \times m$ matrix with real entries. The set of all such matrices will be denoted by $\rr^{l\times m}$. Let $s_1,\ldots, s_k$ be the singular values of $M$, where $k = \min\{l,m\}$.  Recall that the Hilbert--Schmidt norm of $M$ is defined as
\[
\|M\|_{\mathrm{HS}} := \biggl(\sum_{i=1}^l \sum_{j=1}^m \mu_{ij}^2\biggr)^{1/2} = \biggl(\sum_{i=1}^k s_i^2\biggr)^{1/2}\,,
\]
and that the nuclear norm (or trace norm) of $M$ is defined as $\|M\|_*:= \sum_{i=1}^k s_i$. 

Let $Y$ be an $l\times m$ random matrix with independent entries, whose $(i,j)^{\mathrm{th}}$ entry $y_{ij}$ is normally distributed with mean $\mu_{ij}$ and variance $\sigma^2$. In other words, $Y$ is a noisy version of $M$, where the noise is Gaussian with equal variance for all entries. The goal is to estimate the entries of $M$ using the data $Y$, when both $M$ and $\sigma$ are unknown. Since we are interested in the values of the individual entries, it makes sense to use quadratic loss. That is, if $\hat{M}$ is an estimate of $M$, the risk of $\hat{M}$ is measured by the quantity~$\ee\|\hat{M}-M\|_{\mathrm{HS}}^2$.

The problem of estimating the entries of a large matrix from incomplete and/or noisy entries has received widespread attention in the last fifteen years. Early work using spectral analysis was done by a number of authors in the engineering literature, for example by \cite{azaretal01} and \cite{achlioptas01}. Recent papers on spectral methods for matrix completion include those of~\cite{ccs},  \cite{kmo10a, kmo10b}, \cite{cst}, \cite{nadakuditi14}, \cite{gavishdonoho14} and \cite{chatterjee15}. 

Early examples of non-spectral methods appeared in \cite{fazel02}, \cite{renniesrebro05} and \cite{rv07}. Recently, there has been a surge of activity around non-spectral matrix completion and estimation, especially by nuclear norm penalization.  The idea was popularized through the works of  \cite{candesrecht09}, \cite{candestao10} and \cite{candesplan10}. Notable recent papers on this topic include those of \cite{mht10}, \cite{negahban}, \cite{klt}, \cite{rohdetsybakov11}, \cite{kol2012}, \cite{dgm13}, \cite{donohogavish14} and \cite{davenport}.

%More extensive reviews of the literature can be found in \cite{mht10} and \cite{chatterjee15}, although new methods are being proposed every day.

We will now use the general theory of Section \ref{general} to construct an estimator of $M$ that is adaptively minimax rate-optimal for matrices with small nuclear norm. The proposed estimator is defined in three steps.
\vskip.2in
\hrule
\begin{enumerate}[\indent Step 1:]
\item Let $Z$ be an $l\times m$ matrix with i.i.d.~$N(0,1)$ entries.
\item Let $\hs := \|Y\|_*/\|Z\|_*$.
\item 
Let $\hat{M} := \argmin\{\|A\|_*: A\in \rr^{l\times m},\, \|Y-A\|_{\mathrm{HS}}^2 \le lm\hs^2\}$. 
%\item Let $u$ be the Euclidean projection of $\TX\beta^*$ on to the column space of $X$.
%\item Let $\bh := \argmin\{|\beta|_1:\beta\in \rr^p,\, u=X\beta\}$.
\end{enumerate}
\hrule
\vskip.2in
The following theorem gives an upper bound on the risk of $\hat{M}$ under quadratic loss. 
\begin{thm}\label{matthm}
Let $M$, $\sigma$ and $\hat{M}$ be as above. Let
\[
s := \frac{\|M\|_*(\sqrt{l}+\sqrt{m})}{lm\sigma}
\]
Then 
\[
\frac{\ee\|\hat{M}-M\|_{\mathrm{HS}}^2}{lm\sigma^2} \le Cs + C s^2 + C\sqrt{\frac{1}{lm}}
\]
and 
\[
\ee\biggl(\frac{\hs}{\sigma}-1\biggr)^2 \le Cs^2 + \frac{C}{lm}\,,
\]
where $C$ is a universal constant. 
\end{thm}
The purpose of dividing the error by $lm\sigma^2$ is to compare the risk of $\hat{M}$ with the risk of the naive estimator $Y$, which is equal to $lm\sigma^2$. 

Just like Theorem \ref{mainthm}, the proof of this result is a direct application of Theorem \ref{genthm}, by treating $l\times m$ matrices as vectors in $\rr^{lm}$, and taking both $K$ and $\widetilde{K}$ to be the nuclear norm. The details are in Section \ref{proofs}. 

The next theorem shows that in regions where the nuclear norm is neither too small nor too large, $\hat{M}$ is an adaptively minimax rate-optimal estimator. 
\begin{thm}\label{matlow}
Let $l\le m$ be two positive integers. Take any $\delta \ge 0$ and let 
\[
s := \frac{\delta(\sqrt{l}+\sqrt{m})}{lm\sigma}\,.
\]
Suppose that $2/l\le s \le 1$. Consider the setting of Theorem \ref{matthm}. Let $\tilde{M}$ be any estimate of $M$ based on $Y$. Then there exists an $l\times m$ matrix $M$ with $\|M\|_*\le \delta$, such that if this is the true $M$, then 
\[
\frac{\ee\|\tilde{M}-M\|_{\mathrm{HS}}^2}{lm\sigma^2}\ge C s\,,
\] 
where $C$ is a positive universal constant. 
\end{thm}
The asymptotic minimax risk with respect to the nuclear norm in the Gaussian matrix estimation problem was evaluated by \cite{donohogavish14}. In an earlier work, \cite{dgm13} showed that matrix estimation by nuclear norm penalization achieves the minimax risk asymptotically if the penalty parameter is optimally tuned. This tuning, however, would require knowledge about $\sigma^2$. If the elements of $Y$ are uniformly bounded (with a known bound) instead of Gaussian, the USVT estimator of \cite{chatterjee15} is minimax rate-optimal with respect to the nuclear norm. However, for Gaussian entries with unknown $\sigma^2$, there exists no minimax rate-optimal estimator in the literature, other than the one proposed in this section. The proof of Theorem \ref{matlow} is given in Section \ref{proofs}.

\section{Simulation results}\label{simulsec}
This section contains simulation results for the regression estimator proposed in Section \ref{regsec}. For simplicity, the entries of the design matrices were chosen to be i.i.d.~standard Gaussian random variables. The value of $\sigma$ was varied, as was the parameter vector $\beta_0$. The results were compared with the corresponding results from Lasso with 10-fold cross-validation. The output is tabulated in Table \ref{simul}.

\begin{footnotesize}
\begin{table}[t]\centering
\caption{Simulation results comparing the proposed estimator and Lasso with 10-fold cross-validation. The design matrices were constructed with i.i.d.~standard Gaussian entries, to which an intercept term was added. The reported values are averages over $50$ simulations in each case.}
\begin{tabular}{llllrrrrrrr}
\toprule
& & & & \multicolumn{3}{c}{Proposed estimator} & & \multicolumn{3}{c}{Lasso with cross validation}\\
\cmidrule{5-7}
\cmidrule{9-11}
& & & & Average & Average & Average & & Average & Average & Average\\
& & & & \# true  & \# false &  prediction & & \# true  & \# false  & prediction \\
$n$ & $p$ & $\sigma$ & $\ee(y|x)$  & positives & positives & error & & positives & positives & error\\
\midrule
100 & 1000 & 2 & $x_1+x_2$ & $1.42$ & $1.74$ & $1.28$ & & $1.92$ & $12.84$ & $0.77$ \\
100 & 1000 & 3 & $x_1+x_2$ & $0.68$ & $2.00$ & $1.79$ & & $1.18$ & $12.96$ & $1.76$\\
200 & 1000 & 2 & $x_1+x_2-x_3$ & $2.72$ & $0.66$ & $1.51$ & & $3.00$ & $22.62$ & $0.56$\\
200 & 1000 & 3 & $x_1+x_2-x_3$ & $2.16$ & $3.28$ & $1.89$ & & $2.88$ & $17.28$ & $1.27$\\
300 & 300 & 2 & $x_1+2x_2$ & $2.00$ & $12.14$ & $0.25$ & & $2.00$ & $11.48$ &$0.22$ \\
300 & 300 & 3 & $x_1+2x_2$ & $1.94$ & $7.98$ & $0.75$ & & $2.00$ & $11.54$ & $0.52$\\
400 & 4000 & 2 & $x_1+x_2+x_3+x_4$ & $4.00$ & $0.18$ & $1.27$ & & $4.00$ & $31.62$ & $0.38$ \\
400 & 4000 & 3 & $x_1+x_2+x_3+x_4$ & $3.56$ & $1.40$ & $2.16$ & & $4.00$ & $33.56$ & $0.96$ \\
\bottomrule
\end{tabular}
\vskip.2in
\label{simul}
\end{table}
\end{footnotesize}

The table shows that the proposed estimator generally has higher prediction error than Lasso with 10-fold cross-validation. However, the estimator appears to be doing a better job at model selection than the Lasso: it returns a far smaller number of false positives, while detecting the true positives at a rate that is comparable with the Lasso. The large number of false positives returned by the Lasso is considered to be a problematic feature. In the examples that are tabulated in Table \ref{simul}, two to four covariates were included in each model. The Lasso with 10-fold cross-validation typically selected 15 to 25 covariates, whereas the proposed algorithm typically selected less than seven covariates, and usually succeeded in selecting all or most of the relevant covariates.  The most striking example from the table is the following: $n=400$, $p= 4000$, $\sigma = 2$, and $\ee(y|x) = x_1+ x_2+x_3+x_4$. In this example, there are four relevant covariates. Simulations were run 50 times. In all instances, both the proposed estimator and the Lasso with 10-fold cross-validation succeeded in selecting the four relevant covariates. On the other hand, the proposed estimator rarely selected more than one or two irrelevant covariates, whereas the Lasso selected approximately 32 irrelevant covariates on average. Incidentally, the tendency of the Lasso and other high dimensional regression algorithms for selecting large numbers of irrelevant variables is well known among practicing statisticians; a recent paper where this has been noted is \cite{gselletal13}.

Recall that Theorem \ref{mainthm} is a result about the prediction error of the proposed estimator. The simulation results presented in Table \ref{simul} suggest that it would be interesting to have a counterpart of Theorem~\ref{mainthm} that analyzes the model selection property of the estimator. % and gives a mathematical proof of good performance. %For example, it is worth investigating whether the proposed estimator satisfies some kind of oracle inequality up to a constant factor, like the square-root Lasso. 

\section{Proofs}\label{proofs}
\subsection{Proof of Theorem \ref{hsthm}} 
Let $K$ be a norm on $\rr^n$ and $K^\circ$ be its dual norm. An easy consequence of the definition \eqref{dualdef} of the dual norm is that for any $x$ and $y$, 
\[
x\cdot y \le K^\circ(x) K(y)\,.
\]
In particular, 
\begin{equation}\label{csineq}
K(x)K^\circ(x) \ge \|x\|^2\,,
\end{equation}
where $\|x\|$ is the Euclidean norm of $x$. An important result about the dual norm is that  for any norm $K$, 
\begin{equation}\label{kcc}
K^{\circ\circ} = K\,.
\end{equation}
For a proof, see Theorem 15.1 in \cite{rockafellar70}. Another standard result that we will use is the Hilbert projection theorem, which says that any point in $\rr^n$ has a unique Euclidean projection on to a given closed convex set. Here `Euclidean projection' means a point in the convex set that is closest to the given point. 
\begin{lmm}\label{projlmm}
Let $K$ be a norm on $\rr^n$. Take any $x\in \rr^n$ and for each $L\ge 0$, let $w_L$ be the Euclidean projection of $x$ on to the $K$-ball of radius $L$ centered at zero. Then for any $L$ and $L'$,
\[
\|w_L-w_{L'}\|^2 \le |\|x-w_L\|^2 - \|x-w_{L'}\|^2|\,.
\]
\end{lmm}
\begin{proof}
Take any $L\ge L' \ge 0$. Since $K(w_{L'})\le L$ and $w_L$ is the Euclidean projection of $x$ on to the $K$-ball of radius $L$ centered at zero, therefore for each $t\in [0,1]$,
\[
\|x-(tw_{L'}+ (1-t)w_L)\|^2 \ge \|x-w_L\|^2\,.
\]
This can be rewritten as 
\[
t^2\|w_L - w_{L'}\|^2 + 2t(x-w_L)\cdot (w_L - w_{L'})\ge 0\,.
\]
Dividing throughout by $t$ and letting $t\ra 0$ gives the inequality
\[
(w_{L'} - w_{L})\cdot(x-w_L)\le 0\,.
\]
Consequently,
\begin{align*}
\|w_L - w_{L'}\|^2 &= \|(x-w_{L'}) - (x-w_{L})\|^2 \\
&= \|x-w_{L'}\|^2 + \|x-w_L\|^2 - 2(x-w_{L'})\cdot(x-w_L)\\
&= \|x-w_{L'}\|^2 - \|x-w_L\|^2 + 2(w_{L'} - w_{L})\cdot(x-w_L)\\
&\le \|x-w_{L'}\|^2 - \|x-w_L\|^2\,.
\end{align*}
This completes the proof in the case $L\ge L'$. The case $L< L'$ is treated by exchanging $L$ and $L'$ in the above argument.
\end{proof}
\begin{proof}[Proof of Theorem \ref{hsthm}]
For each $L\ge 0$, let $\hm_L$ be the Euclidean projection of $Y$ on to the $K$-ball of radius $L$ centered at zero.

Suppose that for a particular realization of $Y$, $\hm \ne 0$.  Let $\hl:=K(\hm)$. Take any $\nu$ such that $K(\nu)\le \hl$ and $\|Y-\nu\|^2< n\hs^2$. Since $\hm\ne 0$, therefore $0$ is outside the closed Euclidean ball of radius $\sqrt{n} \hs$ centered at $Y$. On the other hand $\nu$ lies in the interior of this ball. Therefore the chord connecting $\nu$ and $0$ contains a point $\nu'$ that lies on the boundary of this ball. This point satisfies $K(\nu')< \hl$, which is impossible by the definition of $\hl$. Therefore, any $\nu$ with $K(\nu)\le \hl$ must satisfy $\|Y-\nu\|^2\ge n\hs^2$. In other words, 
\begin{equation}\label{hmhl}
\hm = \hm_{\hl}\,.
\end{equation}
The argument also shows that 
\begin{equation}\label{ymyl}
\|Y-\hm\|^2 = n\hs^2\,,
\end{equation}
for otherwise the chord connecting $\hm$ and $0$ would contain a point that would give a contradiction to the definition of $\hm$. 

Fix $L=K(\mu)$ for the rest of the proof. Let  
\[
\hs_L^2 := \frac{\|Y-\hm_L\|^2}{n}\,.
\]
Then by Lemma \ref{projlmm} and the identities \eqref{hmhl} and \eqref{ymyl}, 
\begin{align*}
\frac{\|\hm_L - \hm\|^2}{n} &= \frac{\|\hm_L - \hm_{\hl}\|^2}{n}\nonumber \\
&\le \frac{1}{n}|\|Y-\hm_L\|^2-\|Y-\hm_{\hl}\|^2|\nonumber\\
&=  \frac{1}{n}|\|Y-\hm_L\|^2-\|Y-\hm\|^2|= |\hs_L^2 - \hs^2|\,.
\end{align*}
Therefore,
\begin{align}
\frac{\|\hm_L-\hm\|^2}{n}\le |\hs_L^2 - \sigma^2| + |\hs^2-\sigma^2|\,.\label{mainineq}
\end{align}
Note that we derived this inequality under the assumption that $\hm\ne 0$. Next, suppose that $\hm = 0$. Then $ \|Y\|^2\le n\hs^2$ and $\hm = \hm_0$. Therefore by Lemma~\ref{projlmm},
\begin{align*}
\frac{\|\hm_L - \hm\|^2}{n} &= \frac{\|\hm_L - \hm_0\|^2}{n}\\
&\le \frac{1}{n}|\|Y-\hm_L\|^2-\|Y-\hm_0\|^2|\\
&= \frac{\|Y\|^2}{n} - \hs_L^2 \le \hs^2 - \hs_L^2\,.
\end{align*}
Thus, \eqref{mainineq} holds even when $\hm = 0$. Let
\[
s^2 := \frac{\|Y-\mu\|^2}{n}\,.
\]
Then note that 
\begin{align*}
|\hs_L^2 - \sigma^2 | &\le |\hs_L^2 - s^2| + |s^2-\sigma^2| \\
&= \frac{1}{n} (\|Y-\mu\|^2 - \|Y-\hm_L\|^2) + |s^2-\sigma^2|\\
 &= \frac{1}{n} (2(Y-\mu)\cdot (\hm_L-\mu) - \|\mu-\hm_L\|^2) + |s^2-\sigma^2|\\
 &\le \frac{2}{n} (Y-\mu)\cdot (\hm_L-\mu) + |s^2-\sigma^2|\,.
\end{align*}
Therefore, 
\begin{align*}
\ee|\hs_L^2 - \sigma^2|&\le \frac{2\sigma}{n} \ee\biggl(\sup_{\nu\,:\,K(\nu)\le 2L}Z\cdot \nu\biggr) + \frac{\sqrt{2}\sigma^2}{\sqrt{n}}\\
&= \frac{4\sigma K(\mu)}{n} \ee(K^\circ(Z)) + \frac{\sqrt{2}\sigma^2}{\sqrt{n}}\,.
\end{align*}
Combining this with \eqref{mainineq} shows that
\begin{equation}\label{step1}
\frac{\ee\|\hm_L-\hm\|^2}{n\sigma^2} \le \frac{4 }{n\sigma}K(\mu)\ee(K^\circ(Z)) + \sqrt{\frac{2}{n}} + \frac{\ee|\hs^2-\sigma^2|}{\sigma^2}\,.
\end{equation}
Next, note that since $K(\mu)= L$ and $\hm_L$ is the Euclidean projection of $Y$ onto the $K$-ball of radius $L$ centered at the origin, 
\begin{align*}
\|Y-\mu\|^2 &\ge \|Y-\hm_L\|^2 \\
&= \|Y-\mu\|^2 + \|\mu-\hm_L\|^2 + 2(Y-\mu)\cdot(\mu-\hm_L)\,,
\end{align*}
which gives
\begin{align*}
\ee\|\mu-\hm_L\|^2 &\le\ee( 2(Y-\mu)\cdot(\hm_L-\mu))\\
&\le 2\sigma \ee\biggl(\sup_{\nu\,: \, K(\nu)\le 2L} Z\cdot \nu\biggr) = 4\sigma K(\mu) \ee(K^\circ(Z))\,.
\end{align*}
Combining this with \eqref{step1} and using the inequality $\|\hm-\mu\|^2\le 2\|\hm-\hm_L\|^2 + 2\|\hm_L-\mu\|^2$ completes the proof of the theorem.
\end{proof}
\subsection{Proofs of Theorems \ref{shthm} and \ref{genthm}}
\begin{lmm}\label{zmm}
Let $K$ be a norm on $\rr^n$ and $K^\circ$ be its dual norm. Let $Z\sim N_n(0,I_n)$. Then for any integer $k\in [1,n)$, 
\[
\ee(K(Z)^{-k})\le \frac{ \ee(K^\circ(Z)^k)}{(n-k)^k}\,.
\]
\end{lmm}
\begin{proof}
Recall that $\|Z\|^2$ is a $\chi^2$ random variable with $n$ degrees of freedom, which has probability density function
\[
f(x) = \frac{x^{n/2 - 1} e^{-x/2}}{\Gamma(n/2) 2^{n/2}} 
\] 
on $[0,\infty)$. Consequently, 
\begin{align*}
\ee(\|Z\|^k) &= \int_0^\infty\frac{x^{(n+k)/2 - 1} e^{-x/2}}{\Gamma(n/2) 2^{n/2}}\, dx\\
&= \frac{\Gamma((n+k)/2)2^{k/2}}{\Gamma(n/2)}\,.
\end{align*}
Similarly, if $1\le k < n$,
\begin{align*}
\ee(\|Z\|^{-k}) &= \int_0^\infty\frac{x^{(n-k)/2 - 1} e^{-x/2}}{\Gamma(n/2) 2^{n/2}}\, dx\\
&= \frac{\Gamma((n-k)/2)2^{-k/2}}{\Gamma(n/2)}\,.
\end{align*}
Therefore by \eqref{csineq} and the independence of $\|Z\|$ and $Z/\|Z\|$, we get
\begin{align*}
\ee(K(Z)^{-k}) &\le \ee(\|Z\|^{-2k} K^\circ(Z)^k)\\
&= \ee(\|Z\|^{-k} K^\circ(Z/\|Z\|)^k)\\
&= \ee(\|Z\|^{-k}) \ee(K^\circ(Z/\|Z\|)^k)\\
&= \frac{\ee(\|Z\|^{-k}) \ee(\|Z\|^kK^\circ(Z/\|Z\|)^k)}{\ee(\|Z\|^k)} \\
&=  \frac{\ee(\|Z\|^{-k}) \ee(K^\circ(Z)^k)}{\ee(\|Z\|^k)} = \frac{\Gamma((n-k)/2)2^{-k}\ee(K^\circ(Z)^k)}{\Gamma((n+k)/2)}\,.
\end{align*}
By the identity $\Gamma(t)=(t-1)\Gamma(t-1)$, we get 
\[
\Gamma((n+k)/2)\ge ((n-k)/2)^k \Gamma((n-k)/2)\,,
\]
which completes the proof. 
\end{proof}
Let $Z\sim N_n(0,I_n)$ and let $\mathcal{V}$ be any measurable subset of $\rr^n$. Let $b := \sup_{v\in \mathcal{V}}\|v\|$ and assume that $b$ is finite. Let $M := \sup_{v\in \mathcal{V}} Z\cdot v$.  
The following concentration inequality for $M$ was proved by~\cite{tis76}, although it follows with slightly worse constants from earlier works of \cite{sudakovtsirelson74} and \cite{borell75}. For any $t\ge 0$, 
\begin{equation}\label{gaussmax}
\pp(|M-\ee(M)|\ge t) \le 2e^{-t^2/2b^2}\,.
\end{equation}
In the familiar version of this inequality, the set $\mathcal{V}$ is assumed to be finite. It is easy to pass to arbitrary bounded measurable $\mathcal{V}$ by approximating $M$ by maxima over finite subsets of $\mathcal{V}$ and observing that $M$ and $\ee(M)$ can be recovered in the limit of such approximations. 

We will later need upper bounds on the moments of $M$. When $\mathcal{V}$ is a finite set, it is easy to give general upper bounds, as follows. 
\begin{lmm}\label{gaussmom}
Let $M$, $b$ and $\mathcal{V}$ be as above. Let $N$ be the size of the set $\mathcal{V}$. Suppose that $3\le N<\infty$. Then for any integer $k\in [1,2\log N]$,
\[
(\ee|M|^k)^{1/k} \le 3b \sqrt{\log N}\,.
\]
\end{lmm}
\begin{proof}
Note that for any even integer $k\ge 2$,
\begin{align*}
\ee(M^k) &\le \sum_{v\in \mathcal{V}} \ee(Z\cdot v)^k\le N b^k (k-1)!!\,.
\end{align*}
Consequently, if $m_k := (\ee|M|^k)^{1/k}$, then for any even $k$,
\begin{align*}
m_k &\le N^{1/k}b k^{1/2} \biggl(1-\frac{1}{k}\biggr)^{1/k}\biggl(1-\frac{3}{k}\biggr)^{1/k}\cdots \biggl(1-\frac{k-1}{k}\biggr)^{1/k}\, .
\end{align*}
The inequality $1-x\le e^{-x}$ implies that the product on the right is bounded by $e^{-1/4}$. On the other hand, by H\"older's inequality, if $k\le l$ then $m_k\le m_l$. Therefore if $k\le 2\log N$ and $\log N \ge 1$, then choosing $l$ to be an even number between $2\log N$ and $4\log N$ (which exists because $2\log N \ge 2$), the above inequality applied to $m_l$ gives
\[
m_k \le m_l \le 2e^{1/4}b \sqrt{\log N}\le 3b\sqrt{\log N}\,.
\]
This completes the proof of the lemma. 
\end{proof}
\begin{proof}[Proof of Theorem \ref{shthm}]
Let $Z' := (Y-\mu)/\sigma$. Note that 
\[
|K(Y) - \sigma K(Z')|\le K(\mu)\,.
\]
Therefore 
\begin{align*}
|\hs-\sigma| &= \frac{|K(Y)-\sigma K(Z)|}{K(Z)}\\
&\le \frac{|K(Y)-\sigma K(Z')|}{K(Z)} + \frac{\sigma|K(Z')-K(Z)|}{K(Z)}\\
&\le \frac{K(\mu)}{K(Z)} + \frac{\sigma|K(Z')-K(Z)|}{K(Z)}\,.
\end{align*}
Thus,
\begin{align*}
&\ee(\hs-\sigma)^2 \le 2K(\mu)^2 \ee(K(Z)^{-2}) + 2\sigma^2 \ee(K(Z)^{-2}(K(Z)-K(Z'))^2) \nonumber \\
&\le 2K(\mu)^2 \ee(K(Z)^{-2}) + 2\sigma^2( \ee(K(Z)^{-4})\ee((K(Z)-K(Z'))^4))^{1/2}\,.
\end{align*}
By Lemma \ref{zmm}, this gives
\begin{align}\label{kzz}
\ee(\hs-\sigma)^2 &\le  \frac{2K(\mu)^2m_2^2}{(n-2)^2} + \frac{2\sigma^2m_4^2(\ee((K(Z)-K(Z'))^4))^{1/2}}{(n-4)^2}\,.
\end{align}
Now recall that by the identity \eqref{kcc},
\[
K(Z) = K^{\circ\circ}(Z) = \sup_{v\,:\, K^\circ(v)\le 1} Z\cdot v\,.
\]
Therefore by the concentration of Gaussian maxima (inequality \eqref{gaussmax}),
\[
\pp(|K(Z)-\ee(K(Z))|\ge t) \le 2e^{-t^2/2a^2}
\]
for each $t\ge 0$. Since $Z'\sim N_n(0,I_n)$, this implies that 
\begin{align*}
\pp(|K(Z)-K(Z')|\ge t) &\le \pp(|K(Z)-\ee(K(Z))|\ge t/2)\\
&\qquad  + \pp(|K(Z')-\ee(K(Z))|\ge t/2)\\
&\le 4e^{-t^2/8a^2}\,.
\end{align*}
Thus,
\begin{align*}
\ee((K(Z)-K(Z'))^4) &\le \int_0^\infty 16t^3 e^{-t^2/8a^2}\, dt\\
&= \int_0^\infty 8u e^{-u/8a^2} \, du = 512 \,a^4\,.
\end{align*}
Substituting this in \eqref{kzz}, we get the desired inequality.
\end{proof}
\begin{proof}[Proof of Theorem \ref{genthm}]
Simply combine Theorems \ref{hsthm} and \ref{shthm}, and use the inequality 
\[
|\hs^2-\sigma^2|\le (\hs-\sigma)^2 + 2\sigma|\hs -\sigma|
\]
and finally H\"older's inequality to bound $\ee|\hs-\sigma|\le (\ee(\hs-\sigma)^2)^{1/2}$. 
\end{proof}
\subsection{Proof of Theorem \ref{mainthm}}
First, suppose that $X$ has rank $n$, so that $Y'=Y$. In this case we will use Theorem \ref{genthm} with $\mu=X\beta_0$. Define two functions on $\rr^n$ as
\begin{align*}
K(x) &:= \min\{|\beta|_1: \beta\in \rr^p,\, x= X\beta\}\,,\\
\TK(x) &:= \min\{|\beta|_1: \beta\in \rr^{p+n},\, x= \TX\beta\}\,.
\end{align*}
It is easy to prove that these are norms, since $X$ and $\TX$ have rank $n$. 
Take any $v\in \rr^n$ such that $K(v)\le 1$. Then there exists $\beta\in \rr^p$ such that $|\beta|_1\le 1$ and $v = X\beta$. Therefore, 
\begin{align*}
Z\cdot v &= Z\cdot X\beta= \sum_{j=1}^p \beta_j (Z\cdot X_j)\le |\beta|_1 \max_{1\le j\le p} |Z\cdot X_j|\,.
\end{align*}
By the definition \eqref{dualdef} of the dual norm $K^\circ$, this shows that
\begin{align*}
K^\circ(Z) \le  \max_{1\le j\le p} |Z\cdot X_j|\,.
\end{align*}
Therefore by Lemma \ref{gaussmom}, 
\begin{align}\label{kcbd}
\ee(K^\circ(Z)) \le 3\gamma\sqrt{n\log p}
\end{align}
provided that $p\ge 3$. Similarly, one has
\[
\TK^\circ(Z) \le  \max_{1\le j\le p} |Z\cdot \TX_j|\,.
\]
Let $\TX_j$ denote the $j^{\mathrm{th}}$ column of $\TX$. By the construction of $\TX$,  $\max_{1\le j\le p+n}\|\TX_j\|/\sqrt{n}=\gamma$. This implies, by Lemma \ref{gaussmom}, that
\begin{align}\label{kcbd2}
(\ee(\TK^\circ(Z)^k))^{1/k} \le 3\gamma\sqrt{n\log (p+n)}
\end{align}
for every integer $k\in [1,2\log p]$. 

Next, take any $v\in \rr^n$ such that $\TK^\circ(v)\le 1$. Since $\TK(\TX_j)$ is clearly $\le 1$ for each $j$, therefore $|v\cdot \TX_j| \le 1$ for every $j$ by the definition of the dual norm. The cases $j=p+1,\ldots, p+n$ for this inequality imply that the components of $v$ are all bounded by $(\sqrt{n}\gamma)^{-1}$. Consequently, $\|v\|\le \gamma^{-1}$. Thus, 
\begin{align}\label{abd}
\sup\{\|v\|:v\in \rr^n,\, \TK^\circ(v)\le 1\}\le \gamma^{-1}\,.
\end{align}
Plugging in the estimates \eqref{kcbd}, \eqref{kcbd2} and \eqref{abd} into Theorem \ref{genthm}, and observing that $K(X\beta_0)$ and $\TK(X\beta_0)$ are both bounded above by $|\beta_0|_1$, we get the statement of Theorem \ref{mainthm} when $\mathrm{rank}(X)=n$. Moreover, Theorem \ref{shthm} gives the desired upper bound on $\ee(\hs-\sigma)^2$.

Next, suppose that $\mathrm{rank}(X)=k<n$. Then there is a $k\times n$ matrix $A$ that maps the column space of $X$ on to $\rr^k$ and preserves inner products. Let $X'' = AX$ and $Y'' = AY'$. Then $Y'' \sim N_k(X''\beta_0, \sigma^2I_k)$. Moreover, $\|Y''-X''\beta\|=\|Y'-X'\beta\|$ for any $\beta\in \rr^p$. Therefore the definition of $\bh$ implies that 
\[
\bh = \argmin\{|\beta|_1:\beta\in \rr^p,\, \|Y''-X''\beta\|^2\le k\hs^2\}\,,
\]
where $\hs := \TK(Y)/\TK(Z)$. Define $K''$ on $\rr^k$ as
\[
K''(x) := \min\{|\beta|_1: \beta\in \rr^p,\, x= X''\beta\}.
\]
Again, it is easy to prove that this is a norm since $X''$ has rank $k$. To complete the proof, apply Theorem \ref{hsthm} with $Y''$ and $K''$ in place of $Y$ and~$K$ and use Theorem \ref{shthm} to get a bound for $\ee|\hs^2-\sigma^2|$ using the estimates \eqref{kcbd2} and \eqref{abd} obtained above. Lastly, note that $\|X''\bh - X''\beta_0\| = \|X\bh-X\beta_0\|$, and multiply the resulting inequality by~$k/n$.

\subsection{Proof of Theorem \ref{matthm}}
We can put this problem into the setting of Theorem \ref{genthm} by letting $n= lm$, and writing elements of $\rr^n$ as $l\times m$ matrices by putting the first $l$ components as the first column, components $l+1$ through $2l$ as the second column, and so on. For an element $x\in \rr^n$ let $M(x)$ denote the corresponding matrix. Define $K(x)$ to be the nuclear norm of the matrix $M(x)$. It is easy to see that this is indeed a norm on $\rr^n$. Taking this $K$ in Theorem \ref{genthm} and $\widetilde{K} = K$, it is easy to see that the estimator $\hat{M}$ is precisely the estimator prescribed by Theorem \ref{genthm} in this setting. 

Recall that the spectral norm of $M$ is defined as 
\[
\|M\| := \max_{1\le i\le k} s_i\,.
\]
Suppose that $A$ is an $l\times m$ matrix with $\|A\|\le 1$. Then 
\begin{align*}
\|A\|_{\mathrm{HS}}^2 \le k \|A\|^2\le k\,.
\end{align*}
Combining this with the well-known fact that the spectral norm is the dual of the nuclear norm (see \cite{hornjohnson91}, page 214), it follows that  the quantity $a$ of Theorem \ref{genthm} is bounded by $\sqrt{k}$. 

Next, let $Z = (z_{ij})_{1\le i\le l\, 1\le j\le m}$ be a matrix of i.i.d.~$N(0,1)$ entries. Then by Proposition 2.4 in \cite{rv10}, $(\ee\|Z\|^r)^{1/r}\le C(r)(\sqrt{l}+\sqrt{m})$ for every $r\ge 1$, where $C(r)$ is a constant that depends only on $r$. The proof is now easily completed by inserting these estimates into the error bounds from Theorem \ref{shthm} and Theorem \ref{genthm}. 
%\[
%\frac{\|M\|_*^2 (\sqrt{l}+\sqrt{m})^2}{l^2m^2} + \frac{32\sqrt{2}\sigma^2(\sqrt{l}+\sqrt{m})^4}{l^2m^2}
%\]

\subsection{Proof of Theorem \ref{matlow}}
Throughout this proof, $C$ will denote any positive universal constant, whose value may change from line to line. 
Let $k:=[ls/2]$. Then $1\le k\le l$. Let $M = (\mu_{ij})_{1\le i\le l, \, 1\le j\le m}$ be an $l\times m$ random matrix whose first $k$ rows consist of i.i.d.\ Uniform$[-\sigma,\sigma]$ random variables. Declare the remaining rows, if any, to be zero. Then note that $M$ has rank $\le k\le ls/2$. Since $\|M\|_*$ is the sum of the singular values of $M$ and $\|M\|_{\mathrm{HS}}^2$ is the sum of squares of the singular values of $M$, and the number of nonzero singular values equals the rank of $M$, the Cauchy--Schwarz inequality gives 
\begin{align*}
\|M\|_* &\le (ls)^{1/2} \|M\|_{\mathrm{HS}} \le (ls/2)^{1/2} (\sigma^2 lms/2)^{1/2} \le \delta\,. 
\end{align*} 
Let $Y = (y_{ij})_{1\le i\le l, \, 1\le j\le m}$ be a matrix such that given $M$, the entries of $Y$ are independent, and $y_{ij}\sim N(\mu_{ij}, \sigma^2)$. Then it is not difficult to show that
\[
\ee(\var(\mu_{ij}\mid Y)) \ge C\sigma^2\,.
\]
On the other hand, since $\tilde{\mu}_{ij}$ is a function of $Y$, the definition of variance implies that
\[
\ee((\tilde{\mu}_{ij}-\mu_{ij})^2\mid Y) \ge \var(\mu_{ij}\mid Y). 
\]
Combining the last two displays, we get
\begin{align*}
\ee\|\tilde{M}-M\|_{\mathrm{HS}}^2 &\ge \sum_{i=1}^{k} \sum_{j=1}^m \ee(\tilde{\mu}_{ij}-\mu_{ij})^2 \ge  Ckm\sigma^2\ge Clm\sigma^2s\,,%\label{mmf}
\end{align*}
which completes the proof of the theorem.

\vskip.2in
\noindent{\bf Acknowledgments.} I thank Rob Tibshirani, Sara van de Geer, Martin Wainwright and Guanyang Wang for many helpful comments. 
\bibliographystyle{plainnat}

\end{document}